\documentclass[reqno]{amsart}

\usepackage[utf8]{inputenc}
\usepackage[english]{babel}
\usepackage{stmaryrd}
\usepackage{amsthm}
\usepackage{amsmath}
\usepackage{amssymb}
\usepackage{enumitem}
\usepackage[dvips]{graphicx}
\usepackage{color}
\usepackage{hyperref}
\usepackage{url}
\usepackage[left=3.5cm, right=3.5cm, paperheight=11.8in]{geometry}
\usepackage{fancyhdr}
\usepackage{mathrsfs}
\usepackage{bm}
\usepackage{stmaryrd}

\newtheorem{theorem}{Theorem}
\newtheorem{lemma}[theorem]{Lemma}
\let\oldlemma\lemma
\renewcommand{\lemma}{\oldlemma\normalfont}

\newtheorem{corollary}[theorem]{Corollary}
\let\oldcorollary\corollary
\renewcommand{\corollary}{\oldcorollary\normalfont}
\newtheorem{prop}[theorem]{Proposition}
\let\oldprop\prop
\renewcommand{\prop}{\oldprop\normalfont}
\newtheorem{rmk}[theorem]{Remark}
\let\oldrmk\rmk
\renewcommand{\rmk}{\oldrmk\normalfont}
\newtheorem{example}[theorem]{Example}
\let\oldexample\example
\renewcommand{\example}{\oldexample\normalfont}

\theoremstyle{definition}

\theoremstyle{remark}

\hypersetup{
    pdftitle={Large gaps in the image of the Euler's function},
    pdfauthor={Paolo Leonetti},
    pdfmenubar=false,
    pdffitwindow=true,
    pdfstartview=FitH,
    colorlinks=true,
    linkcolor=red,
    citecolor=green,
    urlcolor=cyan
}

\renewcommand{\rho}{\varrho}

\begin{document}

\title{Large gaps in the image of the Euler's function}
\author{Paolo Leonetti}
\address{Universit\`a Bocconi di Milano, via Sarfatti $25$, $20100$ Milan, Italy.}
\email{leonetti.paolo@gmail.com}


\begin{abstract}  
The aim of this note is to provide an upper bound of the number of positive integers $\le x$ which can be written as $\varphi(n)$ for some positive integer $n$, where $\varphi$ stands for the Euler's function. The order of magnitude of this estimate, which is roughly $x/\sqrt[4]{\ln x}$, implies that the set of Euler's values contains arbitrarily large gaps.
\end{abstract}
\maketitle
\thispagestyle{empty}

\section{Introduction}\label{sec:introduction}

Let $\varphi$ be the Euler's function, so that $\varphi(n)$ stands for the number of integers smaller or equal than $n$ and coprime with $n$, and denote with $V$ the set of Euler's values, that is, the set of positive integers which can be written as $\varphi(n)$ for some positive integer $n$. Then, we show that the set $V$ contains arbitrarily large gaps, viz.  
\begin{displaymath}
\textstyle \limsup_{n \in \mathbf{N}}\,v_{n+1}-v_n = \infty,
\end{displaymath}
where $(v_1,v_2,\ldots)$ represents the sequence of elements of $V$ in increasing order. Let also $V(x)$ be the number of elements of $V$ smaller than $x$. The existence of arbitrarily large gaps in the set $V$ follows by: 

\begin{theorem}\label{th:main1}
There exists a positive constant $c$ such that for all $x \ge 2$ we have
\begin{displaymath}
V(x) \le c\, \frac{x}{(\ln x)^{.2587966}}.
\end{displaymath}
\end{theorem}

Indeed, it implies that $V(x)=o(x)$, hence $v_{n+1}-v_n$ cannot be bounded. It is worth noticing that Theorem \ref{th:main1} is just a refinement of the estimate given by Pillai \cite{Pillai}, where he obtained an upper bound with order of magnitude $x/(\ln x)^{\ln 2/e}$. 

As far as the Euler's function is bijective on primes, it is immediate that $V(x)$ is greater than the number of primes smaller than or equal to $x$. Then, according to the Prime Number Theorem (or to the elementary bounds provided by Chebyshev, see e.g. \cite{Apostol}), $V(x)$ has to be a order of magnitude greater than or equal to $x/\ln x$.  A natural question arises: what is the \textquotedblleft correct\textquotedblright\hspace{.3mm} exponent $t$ such that the order of magnitude of $V(x)$ is exactly $x/(\ln x)^t$? In its seminal paper \cite{Erdosphi}, Erd\"{o}s proved that the answer has to be $1$, meaning that the $\varphi$-values in $[1,x]$ are so exceptional that they are just $x/(\ln x)^{1+o(1)}$. Further improvements can be found in \cite{Maier} and \cite{Ford}. At the moment, we still do not know if a natural asymptotic formula exists.

\section{Preparations}\label{sec:preliminaries}

Here and later, $\mathbf{N}$, $\mathbf{P}$, and $\mathbf{R}$ stand for the set of positive integers, primes, and reals, respectively. At this point, we are going to prove some preliminary lemmas: all of them are standard results in number theory. Therefore, the reader who just wants to read the proof of Theorem \ref{th:main1} can skip directly to Section \ref{sec:proof}. 

\begin{lemma}\label{pie}
Let $\mu$ represent the M\"{o}bius function, which is the arithmetical function defined by $\mu(1)=1$, $\mu(n)=(-1)^k$ if $n\ge 2$ is squarefree, where $k$ is the number of distinct prime factors of $n$, and $\mu(n)=0$ otherwise. Let also $m$ be a positive integer and $N$ a non-empty subset of $\mathbf{N}$. Then, the number of integers which belong to $N$, smaller than or equal to $x$, and coprime with $m$ is equal to
\begin{displaymath}
\sum_{d \mid m}{\mu(d) \left|N \cap d\mathbf{N} \cap [1,x]\right|}.
\end{displaymath}
\end{lemma}
\begin{proof}
It is easy to see that the arithmetical function $n\mapsto \sum_{d\mid n}\mu(d)$ is multiplicative, so that it is equal to $1$ if $n=1$ and $0$ otherwise (indeed, it is enough to verify it for the prime powers); then the required sum is equal to
\begin{displaymath}
\sum_{\substack{n \in N\cap [1,x], \\ \text{gcd}(n,m)=1}}{1}=\sum_{n \in N\cap [1,x]}\left(\sum_{d\mid \text{gcd}(n,m)}{\mu(d)}\right).
\end{displaymath}
Let us evaluate this value by double counting: fix a positive integer $d$ which divides $\text{gcd}(n,m)$. In particular the set of possible values of $d$ will be a subset of divisors of $m$, so that the element $\mu(d)$ will be counted exactly $|\{n \in N \cap [1,x]\colon d\mid n\}|$ times. It follows that the required number can be rewritten as
\begin{displaymath}
\sum_{d\mid m}\left(\mu(d)\sum_{n \in N\cap [1,x], d\mid n}{1}\right),
\end{displaymath} 
which is equivalent to the claim. 
\end{proof}

This first result allows us to deduce the rather famous formula for the computation of $\varphi(n)$. Indeed, setting $x=n$ and $N$ equal to set of positive integers $\le n$, we obtain that
\begin{displaymath}
\varphi(n)=\sum_{d\mid n}{\mu(d)|N\cap d\mathbf{N}\cap [1,n]|}=n\sum_{d\mid n}{\frac{\mu(d)}{d}}=n\prod_{p\mid n}{\left(1-\frac{1}{p}\right)}.
\end{displaymath}


\begin{lemma}\label{lemma:abel}
\textbf{(Abel summation).} Let $(\lambda_n)_{n\in \mathbf{N}}$ be a strictly increasing and unbounded sequence of positive reals, $(a_n)_{n\in \mathbf{N}}$ a sequence of complex numbers, and $f$ a differentiable complex-valued function defined on positive reals. Define also $\alpha(x)=\sum_{\lambda_n \le x}{a_n}$ for all reals $x$ greater than $\lambda_1$. Then
\begin{displaymath}
\sum_{\lambda_n\le x}{a_nf(\lambda_n)}=\alpha(x)f(x)-\int_{\lambda_1}^{x}{\alpha(t)f^\prime(t)\,\text{d}t}.
\end{displaymath}
\end{lemma}
\begin{proof}
Define for convenience $\lambda_0=\alpha(0)=0$. Then for all positive integers $m$
\begin{displaymath}
\begin{split}
\sum_{1\le n\le m}{a_nf(\lambda_n)} & =\sum_{1\le n\le m}{\left(\alpha(\lambda_n)-\alpha(\lambda_{n-1})\right)f(\lambda_n)} \\
                                          & =\alpha(\lambda_m)f(\lambda_m)-\sum_{1\le n\le m-1}{\alpha(\lambda_n)\left(f(\lambda_{n+1})-f(\lambda_n)\right)}.
\end{split}
\end{displaymath}
At this point, let $x$ be a real greater than $\lambda_1$, and define $m$ the greatest integer such that $\lambda_m \le x$, which exists by assumption. Considering that $\sum_{\lambda_n\le x}{a_nf(\lambda_n)}$ is equal to $\sum_{1\le n\le m}{a_nf(\lambda_n)}$ and that the function $\alpha$ is constant in $[\lambda_n,\lambda_{n+1})$, then
\begin{displaymath}
\begin{split}
\sum_{\lambda_n\le x}{a_nf(\lambda_n)} & = \alpha(\lambda_m)f(\lambda_m)-\sum_{1\le n\le m-1}{\alpha(\lambda_n)\left(f(\lambda_{n+1})-f(\lambda_n)\right)} \\
                                          & = \alpha(\lambda_m)f(\lambda_m)-\sum_{1\le n\le m-1}{\alpha(\lambda_n)\int_{\lambda_n}^{\lambda_{n+1}}{f^\prime(t)\,\text{d}t}}    \\
																					& = \alpha(\lambda_m)f(\lambda_m)-\sum_{1\le n\le m-1}{\int_{\lambda_n}^{\lambda_{n+1}}{\alpha(t)f^\prime(t)\,\text{d}t}}.   
\end{split}
\end{displaymath}																					
Hence it turns out that this sum is equal to
\begin{displaymath}
\alpha(\lambda_m)f(\lambda_m)-\int_{\lambda_1}^{\lambda_{m}}{\alpha(t)f^\prime(t)\, \text{d}t},   
\end{displaymath}
which is equivalent to the claim.
\end{proof}

It is worth noticing that, in the setting of Stieltjes integral, the summation takes the innocuous form of partial integration, that is why this result is commonly known as \textquotedblleft partial summation.\textquotedblright\hspace{.3mm} Indeed, under the assumption of Lemma \ref{lemma:abel}, the required sum can be directly rewritten as
\begin{displaymath}
\int_{\lambda_1}^x{f(t)\,\text{d}\alpha(t)}=\left[ \text{ }\alpha(t)f(t)^{\text{ }}_{\text{ }} \right]_{\lambda_1}^x - \int_{\lambda_1}^x{\alpha(t)f^\prime(t)\,\text{d}t}.
\end{displaymath}

For convenience, from here later, we are going to use the Bachmann-Landau notations: given real-valued functions $f, g \colon \mathbf{R}\to \mathbf{R}$, the big-Oh $f(x)=\mathcal{O}(g(x))$ stands for the existence of a constant $c$ such that (the absolute value of) their ratio is upper bounded by $c$ whenever $x$ is sufficiently large, that is $\limsup_{x\to \infty}{\left|f(x)/g(x)\right|} \le c$. Moreover, the little-oh $f(x)=o(g(x))$ means that $|f(x)|$ is definitively arbitrarily smaller than $|g(x)|$, i.e. $\lim_{x\to \infty}{\left|f(x)/g(x)\right|}=0$.


\begin{lemma}\label{lemma:omegabound} 
Let $k$ be a positive integer and $c$ a positive constant. Then the number of integers $n$ smaller than $x$ with exactly $k$ distinct prime factors is smaller than $cx$ whenever $x$ is sufficiently large.
\end{lemma}
\begin{proof}
For each positive integer $k$, let $\varrho_k(x)$ be the number of positive integers $n$ smaller than or equal to $x$ with exactly $k$ distinct prime factors. In terms of Bachmann-Landau notations, the statement is equivalent to $\varrho_k(x)=o(x)$ and it would be sufficient to prove that 
\begin{equation}\label{eq:claim1}
\varrho_k(x)=\mathcal{O}\left(\frac{x}{\ln x} \, (\ln \ln x)^{k-1}\right).
\end{equation}
Let us show this claim by induction, starting from the case $k=1$. For each prime $p$ smaller than or equal to $x$, the number of powers of $p$ in $[1,x]$ is exactly $\lfloor \text{log}_p(x)\rfloor$. Then, summing over all primes we obtain
\begin{displaymath}
\varrho_1(x)= \sum_{p\le x}\left\lfloor \frac{\ln x}{\ln p}\right\rfloor =\mathcal{O}\left(\ln x\sum_{p\le x} \frac{1}{\ln p}\right).
\end{displaymath} 
Setting $f(t)=\frac{1}{\ln t}$ in Lemma \ref{lemma:abel}, we evaluate the sum $\sum_{p\le x}{\frac{1}{\ln p}}$ so that
\begin{displaymath}
\varrho_1(x)=\mathcal{O}\left(\pi(x)+\ln x\int_2^x{\frac{\pi(t)}{t\ln^2 t}\, \text{d}t}\right),
\end{displaymath}
where $\pi(t)$ stands for the number of primes $\le t$. Considering that $\pi(t)=\mathcal{O}\left(\frac{t}{\ln t}\right)$ by the Prime Number Theorem, we have that the argument of the integral is equal to $\mathcal{O}\left(\frac{1}{\ln^3 t}\right)= \mathcal{O}\left(\frac{1}{\ln^2 t}\right)=\mathcal{O}\left(\frac{\mathrm{d}}{\mathrm{d} t}\frac{t}{\ln^2 t}\right)$. Therefore
\begin{displaymath}
\varrho_1(x)=\mathcal{O}(\pi(x))+\mathcal{O}\left(\ln x\int_2^x{\frac{\mathrm{d}}{\mathrm{d} t}\frac{t}{\ln^2 t}\, \text{d}t}\right)=\mathcal{O}(\pi(x)).
\end{displaymath}

At this point, suppose that the claim \eqref{eq:claim1} holds for a positive integer $k$, and let us prove that $\varrho_{k+1}(x)=\mathcal{O}\left(\frac{x}{\ln x}(\ln \ln x)^{k}\right)$. For each positive integer $\alpha_0$ smaller than $\ln x$, consider the numbers of the form $p_0^{\alpha_0} \cdot (p_1^{\alpha_1}\cdots p_k^{\alpha_k})$ not greater then $x$, such that $p_1<\cdots<p_k$, and $p_0^{\alpha_0+1}\le x$. Notice that each number $n\le x$ with $\omega(n)=k+1$ can be expressed in such form at least $k+1$ times, depending on the position of the first factor. It follows that 
\begin{displaymath}
(k+1)\varrho_{k+1}(x) \le \sum_{n \in \mathbf{N}\cap [1,\ln x]}\sum_{p^{n+1}\le x}{\varrho_k\left(\frac{x}{p^{n}}\right)},
\end{displaymath}
and in particular
\begin{displaymath}
\varrho_{k+1}(x)=\mathcal{O}\left(\sum_{n \in \mathbf{N}}\sum_{p^{n+1}\le x}{\varrho_k\left(\frac{x}{p^{n}}\right)}\right).
\end{displaymath}
Since we assumed that $\varrho_k(x)=\mathcal{O}\left(\frac{x}{\ln x}(\ln \ln x)^{k-1}\right)$, we obtain that
\begin{displaymath}
\varrho_{k+1}(x)=\mathcal{O}\left(x(\ln \ln x)^{k-1}\sum_{n \in \mathbf{N}}\sum_{p\le x^{1/(n+1)}}{\frac{1}{p^n \ln (x/p^n)}}\right).
\end{displaymath}
Now, the key observation is that the main contribution of the summation comes from the index $n=1$. For all positive reals $a,b$ such that $a\ge 2b$ we have $\frac{1}{a-b}\le \frac{1}{a}+\frac{2b}{a^2}$, therefore
\begin{displaymath}
\sum_{p^2\le x}{\frac{1}{p(\ln x-\ln p)}}=\mathcal{O}\left(\frac{1}{\ln x}\sum_{p\le \sqrt{x}}{\frac{1}{p}}\right)+\mathcal{O}\left(\frac{1}{\ln^2 x}\sum_{p\le \sqrt{x}}{\frac{\ln p}{p}}\right).
\end{displaymath}


On the one hand, setting $f(x)=\frac{1}{x}$ in Lemma \ref{lemma:abel}, we obtain 
\begin{displaymath}
\sum_{p\le \sqrt{x}}{\frac{1}{p}}\le \frac{\pi(x)}{x}+\int_2^x{\frac{\pi(t)}{t^2}\, \text{d}t}=\mathcal{O}\left(\frac{1}{\ln x}\right)+\mathcal{O}\left(\int_2^x{\frac{\mathrm{d}}{\mathrm{d} t}\ln \ln t\text{ d}t}\right)=\mathcal{O}(\ln \ln x).
\end{displaymath}
On the other hand, with the same argument, setting $f(x)=\frac{\ln x}{x}$ we get
\begin{displaymath}
\sum_{p\le \sqrt{x}}{\frac{\ln p}{p}}\le \frac{\pi(x)}{x}+\int_2^x{\frac{\ln t-1}{t\ln t}\text{d}t}=\mathcal{O}\left(\ln x\right).
\end{displaymath}
Notice that last two upper bounds follow from the first and second Mertens' theorems (see, for example, \cite{Apostol}). Putting these results together, we conclude that
\begin{displaymath}
\varrho_{k+1}(x)=\mathcal{O}\left(\frac{x}{\ln x}\, (\ln \ln x)^{k}\right)+\mathcal{O}\left(x(\ln \ln x)^{k-1}\sum_{n \ge 2}\sum_{p^{n+1}\le x}{\frac{1}{p^n \ln (x/p^n)}}\right).
\end{displaymath}

Since that condition $p^{n+1}\le x$ is equivalent to $\ln (x/p^n) \ge \frac{\ln x}{n+1}$, we have also
\begin{displaymath}
\begin{split}
\sum_{n \ge 2}\sum_{p^{n+1}\le x}{\frac{1}{p^n \ln (x/p^n)}}&=\mathcal{O}\left(\frac{1}{\ln x}\sum_{n \ge 2}\sum_{p^{n+1}\le x}{\frac{n}{p^n}}\right)=\mathcal{O}\left(\frac{1}{\ln x}\sum_{n \ge 2}\, \sum_{m\ge 2}{\, \frac{n}{m^n}}\right) \\
 &=\mathcal{O}\left(\frac{1}{\ln x}\sum_{n \ge 2}\, \frac{n^2}{2^n}\right)=\mathcal{O}\left(\frac{1}{\ln x}\right).
\end{split}
\end{displaymath}
This is enough to conclude that \eqref{eq:claim1} holds also for $k+1$, completing the proof.
\end{proof}    

Observe that the existence of a constant $c_k$ such that $\varrho_{k+1}(x)\le c_k\, \frac{x}{\ln x} (\ln \ln x)^k$ assumes that $k$ is a fixed positive integer, not depending on $x$. As far as we are interested in the case $c_k=c_k(x)$, we can repeat the above proof to obtain  
\begin{equation}\label{eq:estimate1}
\varrho_{k}(x)=\mathcal{O}\left(\frac{x}{\ln x} \cdot \frac{(\ln \ln x)^{k-1}}{(k-1)!}\right).
\end{equation}
It implies that, choosing $k$ as function of $x$, an upper bound of $\varrho_k(x)$ needs an estimate of the order of magnitude of the factorial $(k-1)!$. 

\begin{lemma}\label{lm:stirling} \textbf{(Stirling formula).} There exists a constant $c$ such that
\begin{displaymath}
\ln n!=n\ln n-n+\ln \sqrt{n}+c+\mathcal{O}(1/n).
\end{displaymath}
\end{lemma}
\begin{proof}
Defining $\text{frac}(t)$ the fraction part of $t$, that is $t-\lfloor t\rfloor$, the value $\ln n!$, according to Lemma \ref{lemma:abel}, is equal to
\begin{displaymath}
\sum_{m\le n}{\ln m} =n\ln n-\int_1^n{\frac{\lfloor t\rfloor}{t}\text{d}t}  =n\ln n-(n-1)+\int_1^n{\frac{\text{frac}(t)}{t}\text{d}t}.
\end{displaymath}
Let $g$ be the map defined by $x \mapsto \frac{1}{2}\left(\lfloor x\rfloor +\text{frac}^2(x)\right)$. Then $g$ is continuous and derivable in each non-integer point $x$. It implies that $g^\prime(x)$ is exactly $\text{frac}(x)$, and integrating by parts we obtain
\begin{displaymath}
\int_1^n{\frac{\text{frac}(t)}{t}\text{d}t} =\left[\frac{g(t)}{t}\right]_1^n+\int_1^n{\frac{g(t)}{t^2}\text{d}t} = \int_1^n{\frac{t}{2t^2}\text{d}t}+\int_1^n{\frac{\text{frac}^2(t)-\text{frac}(t)}{2t^2}\text{d}t}.
\end{displaymath}
Moreover, the last integral is convergent to some costant $c$, indeed
\begin{displaymath}
\lim_{n\to \infty}\int_1^n{\left|\frac{\text{frac}^2(t)-\text{frac}(t)}{t^2}\right|\text{d}t}=\mathcal{O}\left( \sum_{m\le n}{\frac{1}{m^2}}\right)=\mathcal{O}(1).
\end{displaymath}
At this point, the proof is complete, noticing that
\begin{displaymath}
\int_1^n{\frac{\text{frac}^2(t)-\text{frac}(t)}{t^2}\, \text{d}t}=c+\int_n^{\infty}{\frac{\text{frac}^2(t)-\text{frac}(t)}{t^2}\, \text{d}t}=c+\mathcal{O}(1/n).
\end{displaymath}
\end{proof}

\noindent Everything is finally ready to prove Theorem \ref{th:main1}. 

\section{Proof of the Theorem \ref{th:main1}}\label{sec:proof}

\begin{proof}
Let $n=p_1^{\alpha_1}\cdots p_k^{\alpha_k}$ be an integer greater than $1$, then $(p_1-1)\cdots (p_k-1)$ divides $\varphi(n)$ according to Lemma \ref{pie}. It implies that, in the worst case, $\varphi(n)$ is divisible by $2^{k-1}$. 

Let $N_k$ be the set of positive integers with at most $k$ distinct prime factors, and $M_k$ its complementary set, so that if $n$ belongs to $M_k$ then $2^k$ divides $\varphi(n)$. Considering that if $X,Y,Z$ are sets such that $X=Y\cup Z$ then $|X|\le |Y|+|Z|$, we get
\begin{displaymath}
\begin{split}
V(x) & \le |\varphi(N_k)\cap [1,x]|+|\varphi(M_k)\cap [1,x]|  		\\
                                 & \le |N_k\cap [1,x]|+|\varphi(M_k)\cap [1,x]| \\
                                & \le \varrho_1(x)+\ldots+\varrho_k(x)+2^{-k}x.
\end{split}
\end{displaymath}

According to Lemma \ref{lemma:omegabound} and the estimate \eqref{eq:estimate1}, the inequality $\varrho_{k}(x) \le \varrho_{k+1}(x)$ holds whenever $k$ is smaller than $\ln \ln x$. Define then $k=k(x)=\lceil c\ln \ln x\rceil$ for some positive real $c$ smaller than $1$, and let us try to minimize the right hand side of the above inequality. Clearly, we would have
\begin{displaymath}
V(x) \le k\varrho_k(x)+2^{-k}x.
\end{displaymath} 

On the one hand, we have that $2^{-k}x$ has order of magnitude $\frac{x}{(\ln x)^{c\ln 2}}$, implying that the greater $c$ is, the smaller its value will be. On the other hand, we have also
\begin{displaymath}
k\varrho_k(x)=\mathcal{O}\left( k\varrho_{k+1}(x)\right)=\mathcal{O}\left(\frac{x}{\ln x} \cdot \frac{(\ln \ln x)^{k}}{(k-1)!}\right).
\end{displaymath}
Taking the logarithm at each side we get
\begin{displaymath}
\ln(k\varrho_k(x))=\ln x-\ln\ln x +k\ln\ln\ln x-\ln(k-1)!+\mathcal{O}(1).
\end{displaymath}
Since we are taking care here only in addends on order of magnitude at least $\ln \ln x$, the term $\ln(k-1)!$ can be substituted with $\ln k!$ as far as their difference is just $\mathcal{O}(\ln\ln\ln x)$. According the approximation provided in Lemma \ref{lm:stirling} we obtain
\begin{displaymath}
\begin{split}
\ln(k\varrho_k(x))&=\ln x-\ln\ln x+k\ln\ln\ln x-k\ln k+k+\mathcal{O}\left(\ln\ln\ln x\right) \\
                  &=\ln x-\ln\ln x(1+c-c\ln c)+\mathcal{O}\left(\ln\ln\ln x\right).
\end{split}
\end{displaymath}

This is enough to conclude that for all $c$ in $(0,1)$ we have
\begin{displaymath}
V(x)=\mathcal{O}\left(\frac{x}{(\ln x)^{1-c+c\ln c}}\right)+\mathcal{O}\left(\frac{x}{(\ln x)^{c\ln 2}}\right)=\mathcal{O}\left(\frac{x}{(\ln x)^{\min\{1-c+c\ln c,c\ln 2\}}}\right).
\end{displaymath}
Since the function $c\mapsto (1-c+c\ln c)$ is strictly decreasing in the interval $(0,1)$, the smaller $c$ is, the smaller $\frac{x}{(\ln x)^{1-c+c\ln c}}$ will be. It implies that the best upper bound estimate is obtained in the case that $1-c+c\ln c=c\ln 2$, that is when $c$ is equal to $c^\star \approx .3733646177$. 

We can finally conclude that
\begin{displaymath}
V(x)=\mathcal{O}\left(\frac{x}{(\ln x)^{c^\star \ln 2}}\right)=\mathcal{O}\left(\frac{x}{(\ln x)^{.2587966}}\right),
\end{displaymath}
which is the optimal result in line with the original Pillai's estimate.
\end{proof}

\section{Acknowledgements}
I am grateful to Salvatore Tringali (Texas A\&M University, Qatar) for having attracted my attention to this interesting problem.

\end{document}